\documentclass[sn-mathphys]{sn-jnl}

\usepackage{amsmath}
\usepackage{amssymb}
\usepackage{amsfonts}
\usepackage{amsthm}
\usepackage{bm}
\usepackage{psfrag}
\usepackage{graphicx}
\usepackage{color}
\usepackage{mathrsfs}
\usepackage{graphics}
\usepackage{subfigure}
\usepackage{lmodern}
\usepackage{anyfontsize}
\usepackage{array}
\usepackage{multirow}
\usepackage{booktabs}
\usepackage{appendix}
\usepackage[mathscr]{euscript}

\theoremstyle{thmstyleone}%
\newtheorem{theorem}{Theorem}
\newtheorem{proposition}[theorem]{Proposition}%
\newtheorem{lemma}[theorem]{Lemma}%

\theoremstyle{thmstyletwo}%
\newtheorem{remark}{Remark}%

\theoremstyle{thmstylethree}%

\def\geq{\geqslant}
\def\leq{\leqslant}

\def\epsilon{\varepsilon}

\DeclareMathAlphabet{\mathpzc}{OT1}{cmtt}{m}{n}

\definecolor{darkred}{rgb}{0.85,0,0}

\definecolor{green}{rgb}{0,0.7,0}

\newcommand{\be}{\begin{equation}}
\newcommand{\ee}{\end{equation}}
\newcommand{\ba}{\begin{array}}
\newcommand{\ea}{\end{array}}
\newcommand{\bea}{\begin{eqnarray*}}
\newcommand{\eea}{\end{eqnarray*}}
\newcommand{\bean}{\begin{eqnarray}}
\newcommand{\eean}{\end{eqnarray}}

\newcommand{\lc}{\mathrel{\raise2pt\hbox{${\mathop<\limits_{\raise1pt\hbox{\mbox{$\sim$}}}}$}}}
\newcommand{\gc}{\mathrel{\raise2pt\hbox{${\mathop>\limits_{\raise1pt\hbox{\mbox{$\sim$}}}}$}}}
\newcommand{\ec}{\mathrel{\raise1pt\hbox{${\mathop=\limits_{\raise2pt\hbox{\mbox{$\sim$}}}}$}}}

\newcommand{\nn}{\nonumber}

\renewcommand{\theequation}{\arabic{section}.\arabic{equation}}

\renewcommand{\thefigure}{\arabic{figure}}%

\numberwithin{equation}{section}

\begin{document}

\title[Energy plus maximum bound preserving RK methods for AC]{Energy plus maximum bound preserving Runge--Kutta methods for the Allen--Cahn equation}


\author[1,2]{\fnm{Zhaohui} \sur{Fu}}\email{fuzhmath@math.ubc.ca}
\author[3,4]{\fnm{Tao} \sur{Tang}}\email{tangt@sustech.edu.cn}
\author[1,4]{\fnm{Jiang} \sur{Yang}}\email{yangj7sustech.edu.cn}

\affil[1]{\orgdiv {Department of Mathematics},\orgname{Southern University of Science and Technology}, \orgaddress{\city{Shenzhen}, \postcode{518055}, \country{China}}}
\affil[2]{\orgdiv {Department of Mathematics}, \orgname{University of British Columbia}, \orgaddress{\city{Vancouver}, \postcode{V6T1Z4}, \country{Canada}}}
\affil[3]{\orgdiv {Division of Science and Technology},\orgname{ BNU-HKBU United International College},\orgaddress{\city{Zhuhai},\postcode{519000},\country{ China}} }
\affil[4]{\orgdiv{SUSTech International Center for Mathematics},\orgname{ Southern University of Science and Technology},\orgaddress{\city{Shenzhen}, \postcode{518055}, \country{China}}}


\abstract{It is difficult to design high order numerical schemes which could preserve both the maximum bound property (MBP) and energy dissipation law for certain phase field equations. Strong stability preserving (SSP) Runge--Kutta methods have been developed for numerical solution of  {\em hyperbolic partial differential equations} in the past few decades, where strong stability means the non-increasing of the maximum bound of the underlying solutions. However, existing framework of SSP RK methods can not handle nonlinear stabilities like energy dissipation law. The aim of this work is to extend this SSP theory to deal with the nonlinear phase field equation of the Allen--Cahn type which typically satisfies both maximum bound preserving (MBP) and energy dissipation law. More precisely, for Runge--Kutta time discretizations, we first derive a general necessary and sufficient condition under which MBP is satisfied; and we further provide a necessary condition under which the MBP scheme satisfies energy dissipation.}




\maketitle

\section{Introduction}\label{Se:intro}
Strong stability preserving Runge--Kutta (SSP-RK) methods have been developed
for numerical solution of  {\em hyperbolic partial differential equations}, starting by Shu \cite{1988} it was observed that some Runge--Kutta methods can be decomposed into {\em convex combinations} of forward Euler steps, and so any convex functional property satisfied by forward Euler will be preserved by these higher-order time discretizations, generally under a different time-step
restriction. This approach was used to develop second- and third-order Runge--Kutta methods that
preserve the strong stability properties of the spatial discretizations developed in that work. In fact,
this approach also guarantees that the intermediate stages in a Runge--Kutta method satisfy the
strong stability property as well. More references in this direction can be found in \cite{GKS2011WSP,TVD,contract,barrier} and a useful
survey article of Gottlieb, Shu and Tadmor \cite{SSP}.

The aim of this work is to extend this SSP theory to deal with the nonlinear phase field equation
of the Allen-Cahn type. To this end, we consider the numerical approximation of the Allen--Cahn equation
\begin{equation} \label{eq1.1}
    u_t =\epsilon \Delta u + \frac{1}{\epsilon} f(u), \qquad x\in \Omega, \; t\in (0,T],
\end{equation}
with initial condition
\begin{equation} \label{eq1.2}
    u(x,0)=u_0(x), \qquad x\in \Omega,
\end{equation}
and the homogeneous Neumann boundary condition or periodic boundary condition, where $\Omega$ is a bounded domain in $R^d$ ($d=1,2,3$).
In this paper, we consider the polynomial double-well potential
\begin{equation} \label{eq1.4}
    F(u)=\frac{1}{4} (1-u^2)^2
\end{equation}
and correspondingly,
\begin{equation} \label{eq1.5}
    f(u)=-F'(u)=u-u^3.
\end{equation}
The solution $u(x,t)$ describes the concentration of two crystal orientations of the same material. In this phase model, $u=1$ represents one orientation and $u=-1$ represents the other. The parameter $\epsilon$ here is the width of the interface between two phases, which is positive and small.

The Allen--Cahn equation can be viewed as the $L^2$ gradient flow of the Ginzburg-Landau free energy
\begin{equation}\label{1.3}
    \mathcal{E}(u)=\int_\Omega \left(\frac{\epsilon}{2} \lvert \nabla u \rvert^2 +\frac{1}{\epsilon} F(u)\right)dx.
    \end{equation}
The $L^2$ gradient flow structure corresponds to an energy dissipation law. This means that the energy is decreasing as a function of time,
\begin{equation} \label{eq1.6}
    \frac{d\mathcal{E}}{dt}=-\int_\Omega \left(\epsilon \Delta u+\frac{1}{\epsilon} f(u)\right)^2 dx \leq 0.
\end{equation}
Another significant feature of the Allen--Cahn equation is its maximum bound preserving (MBP) property in the sense
\begin{equation} \label{eq1.7}
 \Vert u(\cdot, t)\Vert _\infty \le 1
 \end{equation}
 provided that the initial and boundary values are bounded by 1.

The Allen--Cahn equation was originally introduced by Allen and Cahn in \cite{AC} to describe the motion of anti-phase boundaries in crystalline solids. In this context, $u$ represents the concentration of one of the two metallic components of the alloy and
the parameter $\epsilon$  represents the interfacial width, which is small compared to the
characteristic length of the laboratory scale. The homogenous Neumann boundary
condition implies that no mass loss occurs across the boundary walls. Since then, the Allen--Cahn equation has been widely applied to many complicated moving interface problems in materials science and fluid dynamics through a phase-field approach.  Since essential features of
the Allen--Cahn equation are the dissipation law (\ref{eq1.6}) and the MBP property (\ref{eq1.7}), it is important to design numerical schemes satisfying both of them.


There have been many energy-dissipation studies for various numerical schemes for the Allen--Cahn equation, see, e.g., \cite{gradsys,mixedvar,epitaxial1, robust, epitaxy4}, and there have been also many recent works on MBP schemes, see, e.g., \cite{MBPsemipara,ETD1,LiYangZhou,1dDMP}.
 It is shown in \cite{ETD1} that the first- and second-order exponential time differencing (ETD) schemes satisfy MBP (\ref{eq1.7}) unconditionally, while \cite{MBPsemipara} established an abstract framework on MBP for more general semi-linear parabolic equations.
Note that most of the relevant works consider the MBP (\ref{eq1.7}) and the energy dissipation law (\ref{eq1.6}) separately. The exception includes \cite{genAC,ATA,IMEXAC}, but the schemes under consideration are only of first-order accuracy in time.

The present work seems to be the first effort to study high-order time discretizations aiming to preserve both (\ref{eq1.6}) and (\ref{eq1.7}). By applying Shu's SSP-RK theory \cite{1988}, i.e., using the property of the forward Euler method repetitively, we will first obtain a sufficient condition to verify whether a Runge--Kutta method is MBP, and also give a necessary and sufficient condition for $s$-stage $s$-th order MBP-RK methods. Both results will be established by using the so-called Butcher Tableau so the results are easy to verify. Moreover, we will provide a necessary condition to judge whether the MBP-RK solutions preserve the energy dissipation law.
Finally, we will provide some RK2, RK3 and RK4 methods which satisfy both (\ref{eq1.6}) and (\ref{eq1.7}).
A special RK3 method violating the energy dissipation law will be also reported.

The paper is organized as follows. Section \ref{sec2} contains some preliminaries and notations. Section \ref{sec3} analyzes high-order MBP-RK methods to the Allen--Cahn equation by using Shu's theory. We build up the relationship between the Butcher Tableau and the so-called Shu-Osher form \cite{SO1988JCP}. Section \ref{sec4} studies how to preserve the energy dissipation law (\ref{eq1.6}) for the relevant Runge--Kutta methods. Section \ref{sec5} applies the theory of Sections \ref{sec3} and \ref{sec4} to some typical Runge--Kutta schemes for the Allen--Cahn equation. The final section provides  some concluding  remarks.

\section{Interplay between the Butcher Tableau and Shu-Osher form}\label{sec2}

The Runge--Kutta methods are a family of implicit and explicit iterative methods used in temporal discretization for the approximate solutions of ordinary differential equations (ODEs). Consider an ODE system in
time $u'=G(u)$. An explicit Runge--Kutta method is commonly written in the form:
\begin{eqnarray}
        && v_0=u_n,\nonumber \\
        && v_{i}=u_n+\tau \sum_{j=0}^{i-1} a_{ij} G(v_{j}),  \qquad 1\leq i \leq s-1\label{eq2.1a}\\
        && u_{n+1}=u_n+\tau \sum _{j=0}^{s-1} b_j G(v_{j}). \nonumber
\end{eqnarray}
In other words, to specify a particular method, one needs to provide the integer $s$ (the number of stages), and the coefficients $a_{ij}$  (for $1 \le j < i \le s$), $b_j$ (for $j = 1, \cdots, s$) and $c_j$ (for $j = 1, \cdots, s-1$). The matrix $(a_{ij})$ is called the Runge--Kutta matrix, while the $b_j$ and $c_j$ are known as the weights and the nodes \cite{Butcherbook}. These data are usually arranged in a mnemonic device, known as a {\em Butcher tableau} (after John C. Butcher):
\be \label{eq2.2a}
\begin{array}
{l|llllll}
0 & 0\\
c_1 & a_{1,0}& 0\\
c_2 & a_{2,0} & a_{2,1}& 0 \\
... & ...& ...& ...& 0\\
c_{s-1} & a_{s-1,0} & a_{s-1,1} & ...&a_{s-1,s-2}& 0\\
\hline
& b_0 &b_1 &... &...& b_{s-1}
\end{array}
\ee
where
\be \label{eq2.3a}
c_i=\sum_{j=0}^{i-1} a_{ij}, \quad i \geq 1.
\ee
 If we define $a_{sj}=b_j$ for all $j\geq 0$, then the scheme (\ref{eq2.1a}) becomes
\begin{eqnarray}
        && u_n=v_0,\nonumber \\
        && v_{i}=u_n+\tau \sum _{j=0}^{i-1} a_{ij} G(v_{j}), \qquad  1\leq i \leq s \label{eq2.4a} \\
        && u_{n+1}=v_{s}. \nonumber
\end{eqnarray}
We further define a strictly lower-triangular matrix $A_L$ as
\begin{equation} \label{eq2.5a}
  A_L=\left[\begin{array}{lllllll}
    0 &\\
    a_{1,0} & 0  \\
    a_{2,0} &  a_{2,1}&0 \\
    ... &... &...&...\\
    a_{s,0} & a_{s,1}& ... &a_{s,s-1}&0
\end{array}\right].
\end{equation}
On the other hand, the Runge--Kutta method can be written in the Shu--Osher form \cite{SO1988JCP}:
\begin{eqnarray}
        && v_0=u_n, \nonumber \\
        && v_i=\sum _{k=0}^{i-1} \Big(\alpha_{ik} v_k+\tau \beta_{ik} G(v_k) \Big), \qquad 1\leq i \leq s\label{eq2.6a}\\
        && u_{n+1}=v_s, \nonumber
\end{eqnarray}
where consistency condition requires
\begin{equation} \label{consist1}
\sum_{k=0}^{i-1} \alpha_{ik}=1, \qquad 1\leq i \leq s.
\end{equation}
 It is observed in Shu \cite{1988,SO1988JCP} if all coefficients are positive, i.e., $\alpha_{ik}> 0$ and $\beta_{ik}\geq 0$, then the solution can be viewed as convex combinations of forward Euler solutions.
 Based on this theory, the consistency condition (\ref{consist1}) and the positivity conditions $\alpha_{ik}\ge 0$ and $\beta_{ik}\geq 0$ can ensure the Strong Stability Preserving (SSP) properties.

 \begin{proposition} (\cite{1988,SO1988JCP})
 If the following so-called RK-SSP condition is satisfied
 \be \label{SSP-cond}
 \sum_{k=0}^{i-1} \alpha_{ik}=1, \quad 1\leq i \leq s; \qquad  \alpha_{ik}\ge 0, \;\; \beta _{ik}\geq 0, \quad 0\le k < i \le s,
 \ee
(when $\alpha_{ik}=0$ then $\beta_{ik}= 0$), then the Runge--Kutta type method of type (\ref{eq2.6a}) satisfies the SSP condition in the sense that
\begin{equation}
\Vert u_{n+1} \Vert \le \Vert u_n \Vert,
\end{equation}
where $\Vert \cdot \Vert$ is the maximum norm or in the TV semi-norm.
\end{proposition}

Below we explore the relationship between the original form (\ref{eq2.1a}) and the Shu--Osher form (\ref{eq2.6a}). We rewrite the Butcher form with the help of the consistency condition (\ref{consist1}):
\be
        v_i =  v_0+\tau \sum _{j=0}^{i-1} a_{ij} G(v_j)= \alpha_{i0} v_0 + \sum_{j=1}^{i-1} \alpha_{ij} v_0+\tau \sum_{j=0}^{i-1} a_{ij} G(v_j).
\ee
We further use (\ref{eq2.1a}) for the above result to obtain
\begin{eqnarray}
          v_i &=& \alpha_{i0} v_0 +\sum_{j=1}^{i-1} \alpha_{ij} \left(v_j -\tau \sum _{k=0}^{j-1} a_{jk} G(v_k)\right) +\tau \sum \limits_{k=0}^{i-1} a_{ik} G(v_k)\nonumber \\
           &=& \sum_{k=0}^{i-1} \left[\alpha_{ik} v_k+ \tau \left(a_{ik}-\sum_{j=k+1}^{i-1}\alpha_{ij} a_{jk} \right) G(v_k)\right],
           \qquad 1\leq i \leq s. \label{2e10}
    \end{eqnarray}
By defining
\be \label{eq2.10a}
\beta_{ik}=a_{ik}-\sum_{j=k+1}^{i-1}\alpha_{ij} a_{jk}, \qquad 0\le k\le i-1,
\ee
the relationship between the original form (\ref{eq2.1a}) and the Shu--Osher form (\ref{eq2.6a}) is established.

\begin{theorem} \label{thm2.1}
If all elements in the strictly lower-triangular matrix $A_L$ in (\ref{eq2.5a}) are positive, i.e. $a_{ik}>0$ for all $0\leq k <i \leq s$, then there exist coefficients $\alpha_{ij},\beta_{ij}\geq 0$ such that the corresponding explicit Runge--Kutta scheme (\ref{eq2.1a}) satisfies the RK-SSP condition.
\end{theorem}

\begin{proof}
We need to use the given positive elements $a_{ij}$ $(0\le j < i\le s)$ to construct positive coefficient pairs
$(\alpha_{ik}, \beta_{ik})$. Let
\begin{equation}
    \delta= \min_{0\leq k<i \leq s} \frac{a_{ik}}{ \sum_{j=k+1}^{i-1} a_{jk}},
\end{equation}
 and let
 \begin{eqnarray}
  && \alpha _{ij}= \min \left\{\frac{\delta}{2},\; \frac{1}{2(i-1)}\right\}, \qquad \forall\; 1\leq i \leq s,\; 1\leq j <i, \nonumber \\
  && \alpha_{i0}=1-(i-1)\cdot a_{i1}, \qquad 1\leq i \leq s.
  \end{eqnarray}
  It is easy to check that $\alpha_{ik} >0$ for all $0\leq k< i\leq s$. Using the relation (\ref{eq2.10a}) and the fact $\alpha_{ij}
  < \delta$ gives
  \begin{eqnarray}
   \beta_{ik} &=& a_{ik}-\sum_{j=k+1}^{i-1}\alpha_{ij} a_{jk} \nonumber \\
   &>& a_{ik} - \sum_{j=k+1}^{i-1}\delta a_{jk} \nonumber \\
   &=& \sum_{j=k+1}^{i-1} a_{jk} \left( \frac{a_{ik}}{\sum_{j=k+1}^{i-1} a_{jk}} - \delta\right) \ge 0.
   \end{eqnarray}
   Finally, it is easy to observe that
   \be
    \sum_{k=0}^{i-1} \alpha_{ik}=\alpha_{i0} + \sum_{k=1}^{i-1} \alpha_{i1} = 1- (i-1)\cdot \alpha_{i1} +(i-1)\cdot \alpha_{i1}=1.
   \ee
 This completes the proof of the theorem.
 \end{proof}

The above theorem gives a simple sufficient condition which can convert a Runge--Kutta method to be of Shu--Osher type satisfying the RK-SSP condition (\ref{eq2.1a}). Below we derive a sufficient and necessary condition for a wide class of Runge--Kutta method.

\begin{theorem} \label{thm2.2}
An explicit Runge--Kutta method with non-zero sub-diagonal elements satisfies the RK-SSP condition (\ref{SSP-cond}) if and only if all elements in the strictly lower-triangular part of $A_L$ in (\ref{eq2.5a}) are positive.
\end{theorem}

\begin{proof}
The sufficient condition is proved in Theorem \ref{thm2.1}. We now prove the necessary condition. In this case, the explicit Runge--Kutta method with non-zero sub-diagonal elements satisfies the RK-SSP condition (\ref{SSP-cond}). Define $order(a_{ik})=(i-1)*s+k, \; 0\leq k<i \leq s $. If there exist non-positive elements in $A$, we take the first of them in the sense of order, $a_{pq} \le  0$. Since the RK scheme satisfies
(\ref{SSP-cond}), we have
\begin{equation} \label{eqt0}
    \alpha_{ik} \geq 0, \;\;\; \beta_{ik}=a_{ik}-\sum_{j=k+1}^{i-1}\alpha_{ij} a_{jk} \geq 0, \quad  0\le k\le i-1.
\end{equation}
In particular, we have
\begin{equation} \label{eqt1}
    a_{pq}-\sum _{j=q+1}^{p-1}\alpha_{pj} a_{jq} \geq 0.
\end{equation}
As $a_{pq}$ is the first non-positive element in the sense of order, all $a_{jq}$ in the summation above  are all positive. We then have two cases.
\begin{itemize}
\item
 If $a_{pq} <0$, then  it is easy to see a contradiction to (\ref{eqt1}).
 \item
  If $a_{pq}= 0$, as the other $a_{jq}$ is positive then it follows from (\ref{eqt1}) that all $\alpha_{pj}$ in (\ref{eqt1}) are $0$, and in particular, $\alpha_{p,p-1}=0$, which leads to $\beta_{p,p-1}=0$. Note that by (\ref{eq2.10a}) we have $\beta_{p,p-1}=a_{p,p-1}$.
  Consequently, we have $a_{p,p-1}=0$ which contradicts the non-zero sub-diagonal element assumption.
\end{itemize}
This completes the proof of the theorem.
\end{proof}

One direct result is the following proposition.
\begin{proposition}
An $s$-stage $s$th-order explicit Runge--Kutta method satisfies the RK-SSP condition (\ref{SSP-cond}) if and only if all elements in the strictly lower-triangular part of $A_L$ are positive.
\end{proposition}
\begin{proof}
For the $s$-order RK scheme, in order to match the highest order term in the Taylor expansion, we must have
\[
\frac{1}{s!}=a_{1,0}a_{2,1} \cdots a_{s,s-1}
\]
which guarantees all sub-diagonal elements $a_{i,i-1}$ are non-zero.
\end{proof}

The following proposition is given in \cite{Butcherbook,contract}, while Theorem \ref{thm2.2}  provides a different perspective.

\begin{proposition}
There does not exist any 4-stage 4th-order explicit Runge--Kutta method satisfying the RK-SSP condition (\ref{SSP-cond}).
\end{proposition}

\begin{proof}
The only 4th-order RK whose coefficients are all non-negative is the classic RK4 \cite{Butcherbook}, whose Butcher tableau reads
\[
\begin{array}{c|cccc}
     0& 0 \\
     \frac 12 & \frac 12 \\
     \frac 12 & 0 & \frac 12\\
     \frac 12 & 0 & 0 &\frac 12\\
     \hline
        & \frac 16 & \frac 13 & \frac 13 & \frac 16
\end{array}
\]
Note that $a_{21}=a_{30}=a_{31}=0$, i.e., they are not positive. Consequently, the classical RK4 does not satisfy the RK-SSP condition (\ref{SSP-cond}).
\end{proof}

\begin{proposition}
Any irreducible RK method whose elements in the strictly lower-triangular part are all positive can not have order greater than 4.
\end{proposition}

\begin{proof}
It is known that there is no irreducible SSP-RK method which has order greater than 4 \cite{contract,barrier}. If an explicit irreducible Runge--Kutta method has positive  strictly lower-triangular part, then based on Theorem \ref{thm2.1} it must satisfy the RK-SSP condition, which contradicts the existing theory of \cite{contract,barrier}.
\end{proof}

\begin{remark}
Theorems in this section could also be derived by the contractivity theory \cite{contract,ferracina2004}, although the approaches and illustrations are different.
\end{remark}

\section{MBP-RK methods for the Allen--Cahn equation}\label{sec3}

We will use the central finite difference discretization to the Allen--Cahn equation in space.
Without loss of generality, we consider the computational domain $[0,2\pi]$ with the periodic boundary condition and let the space mesh size $ h={2\pi}/{N} $. Denote the grid points as $\{ x_j=jh, j=0,1,...,N-1\}$ and the forward finite difference matrix of $\partial_x$ by $D_1$
\begin{equation}
    D_1=\frac{1}{h} \left[\begin{array}{ccccc}
         1& & & & -1 \\
         -1&1& & &\\
          & \cdots & \cdots & \cdots  & \\
          & & &-1&1\\
    \end{array}\right]_{N\times N}.
\end{equation}
Thus we have the central difference discretization operator $D=-D_1^T D_1$ for the Laplacian $\Delta$. It is well-known that the discrete operator $D$ is of second-order accuracy to approximate the Laplacian operator.

\begin{lemma} \label{label3.1}
Given any vector $\mathbf{v}$ and scalar $\alpha>2$, the following inequality holds:
\begin{equation}
\left\|\left( I+\frac 1{\alpha} h^2 D \right)\mathbf{v}\right\|_\infty\leq \|\mathbf{v}\|_\infty.
\end{equation}
\end{lemma}

\begin{proof}
When $\alpha>2$, $\alpha I+D$ is a tri-diagonal matrix whose elements are all positive. Besides, Note that the sum of each row of $D$ is zero. Consequently, the sum of every row of $\alpha I+h^2 {D}$ equals to constant $\alpha$. Observe that
\begin{equation}
\begin{aligned}
    \|(\alpha I+h^2{D})\mathbf{v}\|_\infty &=\max_j \| a_j \mathbf{v}_{j-1}+b_j \mathbf{v}_j +c_j \mathbf{v}_{j+1}\|_\infty \\
    & \leq (\lvert a_j \rvert+\lvert b_j \rvert+\lvert c_j \rvert)\max_j \|\mathbf{v}_j\|_\infty=\alpha\|\mathbf{v}\|_\infty. \\
\end{aligned}
\end{equation}
This completes the proof.
\end{proof}

\begin{lemma}[\cite{lap1}]
Denote the discrete Fourier transform as $F_N$ and the conjugate transpose as  $(\cdot)^H$, then it holds that
\begin{equation}
    D=F_N^H \Lambda F_N, \quad \Lambda=\text{diag}(\left[ \lambda_0, \cdots,\lambda_{N-1} \right]),
\end{equation}
where $\lambda_j=-(2-2\cos(jh))/{h^2}$ are eigenvalues of $D$.
\end{lemma}

One direct result of this lemma is the following inverse inequality:
Given any vector $\textbf{u}$, it holds that
\begin{equation} \label{eq3-7}
    0\le-\textbf{u}^T D \textbf{u} \leq \frac{4}{h^2} \textbf{u}^T \textbf{u}.
\end{equation}
Note that the above property holds for more general boundary conditions and domains, and in these situations the coefficient 4 in (\ref{eq3-7}) will be replaced by a constant  $C$ depending only on the boundary conditions and the domain.

For simplicity and for ease of demonstrating the main ideas, in this paper we only consider the 1D case. For multi-dimension cases, by using the tensor product for the discrete Laplacian operator in 2D and 3D, results similar to the 1D case can be obtained.

The semi-discrete finite difference discretization of the Allen--Cahn equation reads
\begin{equation} \label{eq2.7}
    \frac{d}{dt} \textbf{u}= \epsilon D \textbf{u} + \frac{1}{\epsilon} f(\textbf{u}) =: G(\mathbf{u}).
\end{equation}

We list following two properties for system (\ref{eq2.7}) resulting from the so-called method-of-line approach.
\begin{itemize}
\item
If the initial value satisfies $\|\mathbf{u}_0\|_{\infty} \leq 1$, then the solution $\mathbf{u}(t)$ given by (\ref{eq2.7}) satisfies the maximum bound preserving (MBP) property:
\begin{equation}
    \|\mathbf{u}(t)\|_{\infty}\leq 1, \quad \forall t\geq 0.
\end{equation}
\item
Let
\begin{equation}
    E_h(\mathbf{u})=\frac{\epsilon}{2} \|D_1 \mathbf{u}\|^2_{l^2}+\frac{1}{4\epsilon}\|1-\mathbf{u}^2\|^2_{l^2}.
\end{equation}
Then the solutions of system (\ref{eq2.7}) satisfy the semi-discrete energy dissipation law
\begin{equation}
    \frac{d}{dt} E_h = - \left\|\frac{d \mathbf{u}}{d t}\right\|^2_{l^2}\leq 0.
\end{equation}
\end{itemize}
Note that the first result can be found in, e.g., \cite{semidscrt}, and the second result can be obtained by taking
 the $L^2$ inner product of (\ref{eq2.7}) with $\frac{d}{dt} \mathbf{u}$.

In this section, we are concerned with MBP Runge--Kutta method for the Allen--Cahn equation. The main strategy is to extend the Shu-Osher theory for the hyperbolic conservation laws to deal with the Allen--Cahn solutions.

\subsection{Forward Euler solution}
In this section we discretize the semi-discrete system in the time direction by applying forward Euler method.

Before providing a useful theorem, we need following simple results, which can be obtained by an elementary proof.

\begin{lemma} \label{lem31a}
For any positive number $a$, if $-4a\leq c\leq  a/2$, then the function $g(x)=ax+c(x-x^3)$ satisfies
 \begin{equation} \label{l1a}
 \lvert g(x)\rvert \leq a, \quad \forall x\in [-1,1] .
\end{equation}
\end{lemma}

The following theorem characterizes the Euler property for the system (\ref{eq2.7}).

\begin{theorem} \label{thm31a}
Consider the ODE system (\ref{eq2.7}).
If $\tau< \tau_0:=\min \{{4h^2}/{\epsilon}, {\epsilon}/{4}\}$, then for any vector $\mathbf{u}$ satisfying $\|\mathbf{u}\|_\infty\leq 1$, we have
\begin{equation}
     \|\mathbf{u}+\tau G(\mathbf{u})\|_\infty \leq 1.
\end{equation}
\end{theorem}

\begin{proof}
Note that
\begin{eqnarray}
        \|\mathbf{u}+\tau G(\mathbf{u})\|_\infty&=& \left \|\mathbf{u}+\tau \left(\epsilon D \mathbf{u} + \frac{1}{\epsilon} f(\mathbf{u})\right)\right\|_\infty\nonumber \\
        &=& \left \|\left(\frac{1}{2} \mathbf{u}+ \tau \epsilon D \mathbf{u}\right)+\left(\frac{1}{2} \mathbf{u} +\frac{\tau}{\epsilon} f(\mathbf{u})\right)\right \|_\infty\nonumber \\
        &\leq& \left \|\frac{1}{2} \mathbf{u}+ \tau \epsilon D \mathbf{u}\right \|_\infty+\left \|\frac{1}{2} \mathbf{u} +\frac{\tau}{\epsilon} f(\mathbf{u})\right \|_\infty.
\end{eqnarray}
Using Lemma \ref{label3.1} and the assumption $\tau< {4h^2}/{\epsilon}$ gives
\[
\left \|\frac{1}{2} \mathbf{u}+ \tau \epsilon D \mathbf{u}\right \|_\infty \le \frac 12.
\]
Using (\ref{eq1.5}), Lemma \ref{lem31a} and the assumption $\tau < \epsilon/4$ yields
\[
\left \|\frac{1}{2} \mathbf{u} +\frac{\tau}{\epsilon} f(\mathbf{u})\right \|_\infty \le \frac 12.
\]
Combining the above three results gives the desired result.
\end{proof}

\begin{remark}
The relationship between the time step and $\epsilon$ comes from here and is inevitable if one wants to solve with explicit methods directly.
\end{remark}


\subsection{MBP-RK methods}

\begin{theorem}
Consider the Runge--Kutta scheme (\ref{eq2.6a}) with $G$ defined by (\ref{eq2.7}). If the SSP-RK property (\ref{SSP-cond})is satisfied, then
\begin{equation} \label{SSP-1}
    \|u^n\|_\infty \leq 1 \;\; \implies \;\; \|u^{n+1}\|_\infty \leq 1
\end{equation}
under the time-step restriction
\begin{equation} \label{time1}
    \tau \leq \tau_{SSP} := \min_{0\leq k <i \leq s} \frac{\alpha_{ik}}{\beta_{ik}}\cdot \tau_0, \quad {\rm with} \;\;\;
    \tau_0=\min \left\{ \frac{4h^2}\epsilon, \frac \epsilon 4 \right\}.
\end{equation}
Note that the ratio above is understood as infinity whenever $\beta_{ik}=0$.
\end{theorem}

\begin{proof}
The proof is based on the original SSP machinery, see, e.g., \cite{SSP,TVD}. In particular, note that for each $i$, we have
\begin{equation}
        \|v_i\|_\infty = \left \|\sum _{k=0}^{i-1} \left(\alpha_{ik} v_k+\tau \beta_{ik} G(v_k)\right)\right\|_\infty
        \leq \left\| \sum_{k=0}^{i-1} \alpha_{ik} \left( v_k+\tau \frac{\beta_{ik}}{\alpha_{ik}} G(v_k)\right)\right\|_\infty.
\end{equation}
Under the assumption (\ref{time1}), we have $\tau \beta_{ik} / \alpha _{ik} \le \tau_0$. Then using Theorem \ref{thm31a} gives
\begin{equation}
 \|v_i\|_\infty \le \sum_{k=0}^{i-1} \alpha_{ik} \left\| v_k+\tau \frac{\beta_{ik}}{\alpha_{ik}} G(v_k)\right\|_\infty
 \le \sum_{k=0}^{i-1} \alpha_{ik}\cdot 1 = 1.
 \end{equation}
 This yields the desired result (\ref{SSP-1}).
\end{proof}

\section{The discrete energy dissipation law}\label{sec4}

The discrete energy is defined as follows
\begin{equation}
    E(\textbf{u})=-\frac{\epsilon}{2} \textbf{u}^T D \textbf{u}+ \frac{1}{\epsilon}\sum_{j=1}^{N} F(\textbf{u}_j).
\end{equation}
For ease of our derivation, we will consider a special class of Runge--Kutta schemes.

\begin{lemma}
Given a Butcher tableau (\ref{eq2.2a}) and the corresponding RK scheme (\ref{eq2.4a}). By suitably arranging coefficients
$\{ c_{ik} \}$, we can obtain a class of RK scheme of the following form:
\begin{equation} \label{eq4.3}
    v_i=\sum _{k=0}^{i-1} p_{ik} v_k  + d_i \tau G(v_{i-1}), \qquad 1\leq i \leq s.
\end{equation}
\end{lemma}

\begin{proof}
We wish to convert the RK formula (\ref{eq2.4a}) into the Shu-Osher format. It follows from (\ref{eq2.10a}) that
\bean
        v_i&=& \sum_{k=0}^{i-1} \left[\alpha_{ik} v_k+ \left(a_{ik}-\sum_{l=k+1}^{i-1} a_{lk} \alpha_{il}\right)\tau G(v_k)\right]\nonumber \\
        &=& \sum_{k=0}^{i-1} \alpha_{ik} v_k+ \sum_{k=0}^{i-2} \left(a_{ik}-\sum_{l=k+1}^{i-1} a_{lk} \alpha_{il}\right)\tau G(v_k)+a_{i,i-1} \tau G(v_{i-1}). \label{eq4.2}
\eean
By forcing the second last term in (\ref{eq4.2}) to $0$, a set of values of $\{\alpha_{ik}\}$ can be determined by $\{ a_{ik} \}$. This will  leave only the last $G$-term in (\ref{eq4.2}). Therefore, we derive $p_{ik}=\alpha_{ik}$ and $d_i=a_{i,i-1}$ and thus the scheme has the unique form (\ref{eq4.3}).
\end{proof}

Note that the consistency condition requires $\sum _{k=0}^{i-1} p_{ik}=1$, but now the coefficients in (\ref{eq4.3}) may be negative.

Before we present the main result of this section, we state the following result whose proof is quite straightforward:
\begin{equation} \label{eq4.4}
    \frac{1}{4} [(a^2-1)^2-(b^2-1)^2]\leq (b^3-b)(a-b)+(a-b)^2, \quad \forall a, b \in [-1, 1].
\end{equation}

\begin{theorem}
For a given SSP-RK solution which has the form (\ref{eq4.3}), we define a upper triangular matrix $\Phi$ given by
\begin{equation}
    \Phi_{ij}=\sum _{k=0}^{i-1}\frac{p_{jk}}{d_j},  \qquad i\leq j
\end{equation}
 and the energy discriminant
 \begin{equation}\label{energyD}
     \Delta_E=\frac{1}{2} (\Phi+\Phi^T).
 \end{equation}
If $\Delta_E$ is positive-definite, then the energy is non-increasing under the time step restriction
\begin{equation} \label{eq4.7}
    \tau \leq \min\left\{ \frac{\lambda}{\frac{1}{\epsilon}+\frac{2 \epsilon}{h^2}}, \; \tau_{SSP}\right\},
\end{equation}
where $\tau_{SSP}$ is the SSP-RK time-restriction given by (\ref{time1}), and $\lambda$ is the smallest eigenvalue of $\Delta_E$.
\end{theorem}

\begin{proof}
Rewrite (\ref{eq4.3}) by using the form of $G$ (for simplicity we drop the $\epsilon$ scale in this proof and notice that here $\textbf{v}_i$ are vectors) and the consistency condition:
\bean
    f(\textbf{v}_i)&=& \frac1{d_{i+1}\tau} \left( \textbf{v}_{i+1}-\sum_{k=0}^i p_{i+1,k} \textbf{v}_k\right)  -D \textbf{v}_i\nonumber \\
    &=& \frac{\textbf{v}_{i+1}-\textbf{v}_i}{d_{i+1}\tau}+\frac1{d_{i+1}\tau}
    \sum_{k=0}^{i-1} p_{i+1,k} (\textbf{v}_i-\textbf{v}_k) -D \textbf{v}_i.
\eean
By using the definition of the potential $F$ and by using (\ref{eq4.4}), we obtain
\begin{eqnarray} \label{eq4.9}
        && \sum_{j=1}^N F({(\textbf{u}_{n+1})}_j)-F({(\textbf{u}_n)}_j)
        =\sum_{i=0}^{s-1} \sum_{j=1}^N F({(\textbf{v}_{i+1})}_j)-F({(\textbf{v}_i)}_j)\nn \\
        &\le&  \sum_{i=0}^{s-1} -(\textbf{v}_{i+1}-\textbf{v}_i)^T f(\textbf{v}_i)+(\textbf{v}_{i+1}-\textbf{v}_i)^2
        =: J_1 + J_2 + J_3,
\end{eqnarray}
where
\begin{eqnarray} \label{eq4.9a}
&& J_1= \sum_{i=0}^{s-1}  \left(1-\frac{1}{d_{i+1}\tau}\right) (\textbf{v}_{i+1}-\textbf{v}_i)^2, \qquad
J_2 = \sum_{i=0}^{s-1} (\textbf{v}_{i+1}-\textbf{v}_i)^T D \textbf{v}_i, \nn \\
&& J_3 = -\sum_{i=0}^{s-1} \frac{\sum_{k=0}^{i-1} p_{i+1,k} (\textbf{v}_i-\textbf{v}_k)^T}{d_{i+1}\tau} (\textbf{v}_{i+1}-\textbf{v}_i). \nn
\end{eqnarray}
It can be easily seen that $J_1$ is simply quadratic, which will be negative for sufficiently small $\tau$.  By denoting $\textbf{w}_i=\textbf{v}_i-\textbf{v}_{i-1}$, we have
\begin{eqnarray}
   J_2 &=&\sum_{i=0}^{s-1} (\textbf{v}_{i+1}-\textbf{v}_i)^T D \left(\frac{\textbf{v}_{i+1}+\textbf{v}_i}{2}-\frac{\textbf{v}_{i+1}-\textbf{v}_i}{2}\right) \nn \\
    &=&\sum_{i=0}^{s-1} \frac{1}{2} \left(\textbf{v}^T_{i+1} D\textbf{v}_{i+1}-\textbf{v}^T_i D \textbf{v}_i-\textbf{w}^T_{i+1} D \textbf{w}_{i+1}\right)\nn \\
    &=&\frac{1}{2} (\textbf{u}^T_{n+1} D \textbf{u}_{n+1}-\textbf{u}^T_n D \textbf{u}_n) -\frac{1}{2} \sum_{i=1}^{s} \textbf{w}^T_i D \textbf{w}_i;
    \end{eqnarray}
and also
\begin{eqnarray}
       J_3 &=& \sum_{i=0}^{s-1} \frac{1}{d_{i+1}\tau} \sum_{k=0}^{i-1} p_{i+1,k}\left(\sum_{m=k+1}^{i} \textbf{w}^T_m\right) \textbf{w}_{i+1} \nn \\
        &=& \sum_{i=1}^{s} \frac{1}{d_i\tau}
        \sum_{k=0}^{i-2}\sum_{m=k+1}^{i-1}p_{ik} \textbf{w}^T_m \textbf{w}_{i} = \sum_{i=1}^{s} \frac{1}{d_i\tau}
        \sum_{m=1}^{i-1}\sum_{k=0}^{m-1}p_{ik} \textbf{w}^T_m \textbf{w}_{i}.
\end{eqnarray}
Combining all these results together, we obtain
\begin{eqnarray}
        E_{n+1}-E_n&=&\sum_{j=1}^N F({(\textbf{u}_{n+1})}_j)-F({(\textbf{u}_n)}_j)-\frac{1}{2} (\textbf{u}^T_{n+1} D \textbf{u}_{n+1}-\textbf{u}^T_n D \textbf{u}_n)\nn \\
        &\leq& \sum_{i=1}^{s}  (1-\frac{1}{ d_i \tau}) \textbf{w}_i^2  - \sum_{i=1}^{s} \frac{1}{d_i \tau}
        \sum_{m=1}^{i-1}\sum_{k=0}^{m-1}p_{ik} \textbf{w}^T_m \textbf{w}_{i}-\frac{1}{2} \sum_{i=1}^{s} \textbf{w}^T_i D \textbf{w}_i\nn \\
        &=&\sum_{i=1}^{s} \textbf{w}_i^2-\frac{1}{\tau} \sum_{m,i=1}^{s}\textbf{w}_m^T \Phi_{mi} \textbf{w}_i-\frac{1}{2} \sum_{i=1}^{s} \textbf{w}^T_i D \textbf{w}_i, \label{eq412}
\end{eqnarray}
where we have defined an upper triangle matrix $\Phi$ by (notice that $\sum_{k=0}^{i-1}p_{ik}=1$)
\begin{equation}
    \Phi_{ij}=\sum_{k=0}^{i-1}p_{jk}/d_j, \qquad i\leq j.
\end{equation}
Consider the energy discriminant $\Delta_E$ defined by (\ref{energyD}).
Recall that we dropped the $\epsilon$ scale in the very beginning. If we keep $\epsilon$ in the derivation, the change of the energy
(\ref{eq412}) becomes
\begin{equation} \label{eq4.14}
    E_{n+1}-E_n\leq \sum_{i=1}^{s} \textbf{w}_i^2-\frac{\epsilon}{\tau} \sum_{i,j=1}^{s}\textbf{w}_i^T \Phi_{ij} \textbf{w}_j-\frac{\epsilon^2 }{2} \sum_{i=1}^{s} \textbf{w}^T_i D \textbf{w}_i.
\end{equation}
If $\Delta_E$ is positive-definite and $\lambda$ is the smallest eigenvalue of $\Delta_E$, then we have
\begin{equation}
    \sum_{i,j=1}^{s}\textbf{w}_i^T \Phi_{ij} \textbf{w}_j= \sum_{i,j=1}^{s}\textbf{w}_i^T {\Delta_E}_{ij} \textbf{w}_j \geq \lambda \sum_{i=1}^{s} \textbf{w}_i^2.
\end{equation}
It follows from (\ref{eq3-7}) that
\begin{equation}
    -\sum_{i=1}^{s} \textbf{w}^T_i D \textbf{w}_i\leq 4 h^{-2} \sum_{i=1}^{s} \textbf{w}_i^2.
\end{equation}
Thus, by (\ref{eq4.14}), to make sure the energy dissipation we only need
\begin{equation}
        1-\frac{\epsilon\lambda}{\tau}+\frac{2 \epsilon^2}{ h^2} \leq 0,
\end{equation}
which is true under the assumption (\ref{eq4.7}).
\end{proof}

\section{Some energy plus MBP RK methods}\label{sec5}
In this section we present some RK2, RK3 and RK4 methods which is maximum bound preserving and energy dissipation. We will also give an RK3 method which is MBP but does not satisfy our condition for the energy dissipation law. All examples in this section are existing schemes, and the special 5-stage 4th-order example comes from \cite{TVD}.

\subsection{An RK2 satisfying energy-dissipation and MBP}
The Butcher Tableau is as follows:
\[
\begin{array}{c|cc}
     0&  \\
     1& 1\\
     \hline
      & \frac 12 &\frac 12
\end{array}.
\]
The corresponding (\ref{eq4.3}) form is given by
\begin{eqnarray}
        &&\textbf{v}_0=\textbf{u}_n,\nonumber\\
        &&\textbf{v}_1=\textbf{v}_0+\tau G(\textbf{v}_0),\nn \\
        &&\textbf{v}_2=\textbf{v}_0+\frac{\tau}{2}G(\textbf{v}_0)+\frac{\tau}{2}G(\textbf{v}_1)=\frac{1}{2} \textbf{v}_0+\frac{1}{2}\textbf{v}_1+\frac{\tau}{2}G(\textbf{v}_1). \label{eq5.1}
\end{eqnarray}
It follows from the theory in Section \ref{sec3} the scheme (\ref{eq5.1}) is MBP.

The energy form coincides with (\ref{eq5.1}) and the energy discriminant is
\begin{equation}
    \Phi=\left( \begin{array}{cc}
        1 & 1 \\
        0 & 2
    \end{array}
    \right),\quad
    \Delta_E=\frac{1}{2} \left(\Phi+\Phi^T \right)=\left( \begin{array}{cc}
        1 & \frac 12 \\
        \frac 12 & 2
    \end{array}
    \right).
\end{equation}
Note that $\Delta_E$ is positive-definite and the smallest eigenvalue of $\Delta_E$ is $\frac{1}{2} (3-\sqrt{2})$. Hence with suitably small time-step, this MBP-RK2 scheme preserves both maximum bound and the energy dissipation law.

\subsection{An RK3 satisfying energy-dissipation and MBP}
Consider the Butcher tableau:
\[
\begin{array}{c|ccc}
     0&  \\
     1& 1\\
     \frac12 & \frac14 &\frac 14 \\
     \hline
      & \frac 16&\frac 16& \frac 23
\end{array}.
\]
The corresponding (\ref{eq4.3}) form is given by
\begin{eqnarray}
        &&\textbf{v}_0=\textbf{u}_n,\nonumber\\
        &&\textbf{v}_1=\textbf{v}_0+\tau G(\textbf{v}_0),\nonumber\\
        &&\textbf{v}_2=\frac{3}{4} \textbf{v}_0+\frac{1}{4} \textbf{v}_1 +\frac{\tau}{4} G(\textbf{v}_1),\\
        &&\textbf{v}_3=\frac{1}{3} \textbf{v}_0+\frac{2}{3}\textbf{v}_2+\frac{2\tau}{3} G(\textbf{v}_2).\nonumber
\end{eqnarray}
The energy discriminant reads
\begin{equation}
    \Phi=\left( \begin{array}{ccc}
        1 & 3 & \frac 12\\
        0 & 4 & \frac 12\\
        0 & 0 & \frac 32
    \end{array}
    \right),\qquad
    \Delta_E=\frac{1}{2} \left(\Phi+\Phi^T \right)=\left( \begin{array}{ccc}
        1 & \frac 32 & \frac 14\\
        \frac 32 & 4 & \frac 14\\
        \frac 14 & \frac 14 & \frac 32
    \end{array}
    \right).
\end{equation}
Note that $\Delta_E$ is positive-definite and the smallest eigenvalue of $\Delta_E$ is about $0.362228$. Again based on the theory in Section \ref{sec4}, with sufficiently small time-step, this MBP-RK3 scheme preserves both maximum bound and the energy dissipation law.

\subsection{An RK3 satisfying MBP but not sure energy-dissipation}
Consider the Butcher tableau:
\[
\begin{array}{c|ccc}
     0&  \\
     1& 1\\
     2 & 1 &1 \\
     \hline
      & \frac 23&\frac 16&\frac 16
\end{array}.
\]
The corresponding (\ref{eq4.3}) form is given by
\begin{eqnarray}
        &&\textbf{v}_0=\textbf{u}_n,\nonumber\\
        &&\textbf{v}_1=\textbf{v}_0+\tau G(\textbf{v}_0),\nonumber\\
        &&\textbf{v}_2= \textbf{v}_1 +\tau G(\textbf{v}_1),\\
        &&\textbf{v}_3=\frac{1}{3} \textbf{v}_0+\frac{1}{2}\textbf{v}_1+\frac{1}{6} \textbf{v}_2+\frac{\tau}{6} G(\textbf{v}_2).\nonumber
\end{eqnarray}
However, It can be shown that the energy discriminant is not positive-definite in this case. Note
\begin{equation}
    \Phi=\left( \begin{array}{ccc}
        1 & 0 & 2\\
        0 & 1 & 5\\
        0 & 0 & 6
    \end{array}
    \right),\quad
    \Delta_E=\frac{1}{2} \left(\Phi+\Phi^T \right)=\left( \begin{array}{ccc}
        1 & 0 & 1\\
        0 & 1 & \frac 52\\
        1 & \frac 52 & 6
    \end{array}
    \right).
\end{equation}
The smallest eigenvalue of $\Delta_E$ is $\frac{1}{2} (7-3\sqrt{6})\approx -0.174$. Thus, this scheme is not guaranteed to decrease the energy by our approach.

\subsection{An 5-stage RK4 satisfying MBP and energy-dissipation}

It is well known that there is no 4-stage 4th-order RK4. We then just consider the following 5-stage RK scheme:
\bean
        && \textbf{v}_0=\textbf{u}_n,\nn \\
        &&\textbf{v}_1=\textbf{v}_0+d_1 \tau G(\textbf{v}_0),\nn \\
        &&\textbf{v}_2=p_{20}\textbf{v}_0+p_{21}\textbf{v}_1+d_2 \tau G(\textbf{v}_1),\nn \\
        &&\textbf{v}_3=p_{30}\textbf{v}_0+p_{32}\textbf{v}_2+d_3 \tau G(\textbf{v}_2), \label{eq5.7} \\
        &&\textbf{v}_4=p_{40}\textbf{v}_0+p_{43}\textbf{v}_3+d_4 \tau G(\textbf{v}_3),\nn \\
        &&\textbf{v}_5=p_{52}\textbf{v}_2+p_{53}\textbf{v}_3+d_{53} \tau G(\textbf{v}_3)+p_{54}\textbf{v}_4+d_{54}\tau G(\textbf{v}_4), \nn
\eean
where
\bea
& d_1=0.391752226571890,  &p_{20}=0.444370493651235, \\
& p_{21}=0.555629506348765, & d_2=0.368410593050371,\\
& p_{30}=0.620101851488403, & p_{32}=0.379898148511597,\\
& d_3=0.251891774271694, & p_{40}=0.178079954393132, \\
& p_{43}=0.821920045606868, & d_4=0.544974750228521, \\
& p_{52}=0.517231671970585, & p_{53}=0.096059710526147, \\
& d_{53}=0.063692468666290, & p_{54}=0.386708617503269,\\
& d_{54}=0.226007483236906. &
\eea
Using the theory of Section \ref{sec3}, it is known that the scheme (\ref{eq5.7}) satisfies MBP. In order to obtain the energy form, we rewrite the last line as
\bean
        \textbf{v}_5&=& p_{52}\textbf{v}_2+p_{53}\textbf{v}_3+d_{53} \frac{\textbf{v}_4-p_{40}\textbf{v}_0-p_{43}\textbf{v}_3}{d_4}+p_{54}\textbf{v}_4+d_{54}\tau G(\textbf{v}_4)\nn \\
        &=& -\frac{d_{53}p_{40}}{d_4} \textbf{v}_0+p_{52}\textbf{v}_2+\left(p_{53}-\frac{d_{53}p_{43}}{d_4}\right)\textbf{v}_3+
        \left(p_{54}+\frac{d_{53}}{d_4}\right)\textbf{v}_4+d_{54}\tau G(\textbf{v}_4).
\eean
Thus the energy discriminant is
\begin{equation}
    \Phi=\left(\begin{array}{ccccc}
        \frac{1}{d_1} & \frac{p_{20}}{d_2} & \frac{p_{30}}{d_3} & \frac{p_{40}}{d_4}&-\frac{d_{53}p_{40}}{d_4} \\
        0 & \frac{1}{d_2} & \frac{p_{30}}{d_3} & \frac{p_{40}}{d_4} & -\frac{d_{53}p_{40}}{d_4} \\
        0 & 0 & \frac{1}{d_3} & \frac{p_{40}}{d_4} & p_{52}-\frac{d_{53}p_{40}}{d_4} \\
        0 & 0 & 0 & \frac{1}{d_4} & p_{52}+p_{53}-\frac{d_{53}}{d_4}\\
        0 & 0 & 0 & 0 & \frac{1}{d_5}
    \end{array}\right)
    ,\quad
    \Delta_E=\frac{1}{2} \left(\Phi+\Phi^T \right).
\end{equation}
The smallest eigenvalue of $\Delta_E$ is about $1.706$. Hence, based on the theory of Section \ref{sec4}, then with sufficiently small time-step this 5-stage RK4 preserves both maximum bound and energy dissipation law.
\section*{Acknowledges}
The work of J.\ Yang is supported by National Natural Science Foundation of China (NSFC) Grant No. 11871264, Natural Science Foundation of Guangdong Province (2018A0303130123),  and NSFC/Hong Kong RRC Joint Research Scheme (NFSC/RGC 11961160718), and the research of Z.\ Yuan and Z.\ Zhou is partially supported by Hong Kong RGC grant (No. 15304420).

\bibliography{ref}

\end{document}